\documentclass{amsart}
\usepackage{amsmath,amsthm,amsfonts,amssymb}
\begin{document} 
\newcommand{\A}{{\mathbb A}}
\newcommand{\B}{{\mathbb B}}
\newcommand{\C}{{\mathbb C}}
\newcommand{\N}{{\mathbb N}}
\newcommand{\Q}{{\mathbb Q}}
\newcommand{\Z}{{\mathbb Z}}
\renewcommand{\P}{{\mathbb P}}
\newcommand{\R}{{\mathbb R}}
\newcommand{\rc}{\subset}
\newcommand{\rank}{\mathop{rank}}
\newcommand{\trace}{\mathop{tr}}
\newcommand{\dimc}{\mathop{dim}_{\C}}
\newcommand{\ddz}{\mathop{\frac{\partial}{\partial z}}}
\newcommand{\Lie}{\mathop{Lie}}
\newcommand{\Spec}{\mathop{Spec}}
\newcommand{\Aut}{\mathop{{\rm Aut}}}
\newcommand{\Auto}{\mathop{{\rm Aut}_{\mathcal O}}}
\newcommand{\alg}[1]{{\mathbf #1}}
\theoremstyle{plain}
\newtheorem{lemma}{Lemma}
\newtheorem{proposition}{Proposition}
\newtheorem{theorem}{Theorem}
\theoremstyle{definition}
\newtheorem{definition}{Definition}
\newtheorem{corollary}{Corollary}
\theoremstyle{remark}
\newtheorem{example}{Example}
\newtheorem*{remark}{Remark}
\title{%
Deformations of Riemann Surfaces
}
\author {J\"org Winkelmann}
\email{jwinkel@member.ams.org}
\address{%
%J\"org Winkelmann \\
Lehrstuhl Analysis II \\
Mathematisches Institut \\
NA 4/73\\
Ruhr-Universit\"at Bochum\\
44780 Bochum \\
Germany
}
%
%\classification{30F20}
\begin{abstract}
We prove that
every Riemann surface not isomorphic to $\P_1$
admits an infinitesimal deformation of
its complex structure. The proof is based on an investigation of
the length of geodesics for the Kobayashi/Poincar\'e metric.
\end{abstract}
\thanks{
{\em Acknowledgement.}
The author was 
supported by the 
SFB/TR 12 ``Symmetries and Universality in mesoscopic systems''.
}
\maketitle
\section{Summary}
The purpose of this article is to prove the following statement:

\medskip
{\em On every orientable surface except the $2$-sphere
$S^2$ there exists at least two non-isomorphic
complex structures.}

\medskip
For compact Riemann surfaces there is a 
classical theory of moduli spaces
which includes a precise description of deformations of these compact
Riemann surfaces. 
This theory easily implies the above statement for compact Riemann surfaces.
This theory extends to complex algebraic curves, i.e., to Riemann surfaces
which can be compactified by adding finitely many points.
However, this theory does not extend to arbitrary
non-compact Riemann surfaces whose topology can be quite complicated,
e.g., the fundamental group may be not finitely generated.
Thus to determine which Riemann surfaces admit deformations,
we need a different approach.

The key idea in our approach is the following: Most Riemann surfaces
are hyperbolic. Hence there is a canonical way to measure the length
of a (real) curve.

Given a hyperbolic Riemann surface $X$, we look for
a closed curve $\gamma:S^1\to X$ for which there exists a bound $c>0$
such that every curve homotopic to $\gamma$ has at least length $c$
where the length is calculated with respect to the Kobayashi pseudometric.
Then we deform the complex structure on some neighbourhood of the
image of $\gamma$ such that the length of $\gamma$ decreases below
$c$. This yields a complex structure non-isomorphic to the original
one.

In this way we obtain the following result.

\begin{theorem}
Let $(X,J)$ be a Riemann surface which is not biholomorphic to $\P_1$.
Then there exists a continuous family of complex structures $J_t$
parametrized by $t\in[0,1]$ such that $J_0=J$ and such that $(X,J_1)$
is not  biholomorphic to $(X,J)$.
\end{theorem}

From this we deduce a related result on discrete subgroups
in $PSL_2(\R)$:

\begin{theorem}
Let $F$ be  free group (possibly not finitely generated)
and let $\rho_0:F\to G=PSL_2(\R)$ be a group
homomorphism which embedds $F$ into $G$ as a discrete subgroup.

Then there exists a continuous family of group homomorphisms 
$\rho_t:F\to G$ ($t\in[0,1]$) such that each $\rho_t$ embedds $F$
as a discrete subgroup in $G$ and such that $\rho_0$ is not
conjugate to $\rho_1$. (I.e.~there is no $g\in G$ such that
$\rho_1(\gamma)=g\cdot \rho_0(\gamma)\cdot g^{-1}\ \forall\gamma\in F$.)
\end{theorem}

\section{Classical Theory}
Since the 19th century it is known that there is a moduli space
for compact Riemann surfaces of a given genus $g$, 
and that the (complex) dimension of this
moduli space equals $3g-3$ for $g\ge 2$ and equals $g$ for $g\le 1$.

This of course implies our statement for the special case of
{\em compact} Riemann surfaces.

There is a generalization of this theory to {\em algebraic complex
curves}, i.e., Riemann surfaces which can be compactified by adding
finitely many points.

If $g$ denotes the genus of the compactification $\bar X$ of such
a complex algebraic curve $X$ and if $d$ denotes the cardinality
of the set $\bar X\setminus X$, then the respective moduli space
has dimension $3g-3+d$ if $g\ge 2$ and every $d\in\N\cup\{0\}$.
For $g=1$ and $d\ge 1$ the dimension of the moduli space is $d$.
For $g=0$ it equals $\max\{0,d-3\}$.

As a consequence, the only complex algebraic curves which do not admit
any non-trivial deformation within the category of algebraic curves
are $\P_1$, $\P_1\setminus\{\infty\}=\C$, $\C^*$ and
$\P_1\setminus\{0,1,\infty\}$. 
(However, as we will see below, the complex structure on $\C$ may
be deformed to that of the unit disc. Hence $\C$ {\em does} admit a
deformation in the category of Riemann surface although it does
not within the category of complex algebraic curves.
Similarily for $\C^*$ and $\P_1\setminus\{0,1,\infty\}$.)

For arbitrary non-compact
Riemann surfaces there is no reasonable theory of moduli spaces.

Already the situation for complex structures on the orientable
surface $\R\times S^1$ is quite intricate:

Every complex structure on $\R\times S^1$ 
defines a Riemann surface biholomorphic
to
\[
A(r,s)=\{z\in\C: r<|z|<s\}
\]
with $0\le r<s\le+\infty$.
The ratio $s/r$ is an invariant of the complex structure.
In this sense the map $(r,s)\mapsto r/s\in[0,1[$
yields almost a moduli space. However, for $r/s=0$ (i.e. $r=0$ or
$s=\infty$) there are {\em two} inequivalent complex structures:
that from $\C^*=A(0,\infty)$ and that from $\Delta^*=A(0,1)$.

Thus one may say that the complex structures on $\R\times S^1$ are
pa\-ra\-me\-tri\-zed by a real-one dimensional non-Hausdorff space:
$[0,1[$ with the zero replaced by a double point.

For arbitrary non-compact Riemann surfaces (in particular those whose
fundamental group is not finitely generated), there is no
reasonable holomorphic classification theory.

(There does exists a {\em topological classification of real surfaces}
though, cf.~theorem of Ker\'ekj\'art\'o., see~\cite{R})

\section{The special cases of $\C$ and $\Delta$}

Here we consider some special cases where (as we will see later) our general
methods can not be applied.

First we show that the complex plane $\C$ can be deformed to
the unit disc $\Delta$ and vice versa.

For this purpose we consider some auxiliary functions.

We define
\begin{align}
\rho_t(r)   &= \left( r + t \arctan\left(\frac{\pi}{2}r\right)\right),\\
\rho^*_t(r) &= \arctan\left( \frac{\pi}{2}r\left(\frac{1}{1+t}\right)\right)
\end{align}
and
\[
\phi_t:z\mapsto z\frac{\rho_t(|z|)}{|z|},\ \ 
\phi^*_t:z\mapsto z\frac{\rho^*_t(|z|)}{|z|}.
\]
We note that for $t\ne 0$, the unit disk $\Delta$ is mapped
bijectively onto $\C$ by $\phi_t$ and bijectively onto a bounded disc
with some radius $s(t)\in\R^+$ by $\phi^*_t$.
Furthermore $\phi_0$ resp.~$\phi_0^*$ maps the unit disk
bijectively onto $\Delta$ resp.~$\C$.

Therefore pulling back the standard complex structure on $\C$ via
$\phi_t$ resp.~$\phi_t^*$ to the unit disc yields a family of
complex structures $J_t$ resp.~$J^*_t$ on $\Delta$ such that
$(\Delta,J_0)$ and $(\Delta,J_t^*)$ ($t\ne 0)$ are biholomorphic
to the unit disc while
$(\Delta,J_0^*)$ and $(\Delta,J_t)$ ($t\ne 0)$ are biholomorphic
to $\C$.

Hence both $\Delta$ and $\C$ admit non-trivial infinitesimal deformations
of the complex structure.

Since $\C^*=\C\setminus\{0\}$ and $\Delta^*=\Delta\setminus\{0\}$ just by
restricting these families of complex structures on $\C$ resp.~$\Delta$
 to the complement
of $\{0\}$ we see that $\C^*$ and $\Delta^*$ can be deformed as well.

Similarily we obtain that the standard complex structure on $\P_1\setminus\{0,1,\infty\}
\simeq\C\setminus\{0,\frac12\}$ can be deformed to $\Delta\setminus\{0,\frac12\}$.

Thus we obtain:
\begin{proposition}\label{case-special}
Let $(X,J)$ be a Riemann surface which is biholomorphic to
$\C$, $\C^*$, $\P_1\setminus\{0,1,\infty\}$, $\Delta$ or $\Delta^*$.

Then there exists a continuous family of complex structures $J_t$
pa\-ra\-me\-tri\-zed by $t\in[0,1]$ such that $J_0=J$ and such that $(X,J_1)$
is not  biholomorphic to $(X,J)$.
\end{proposition}

\section{Ahlfors Schwarz Lemma}
We will use a classical result which is known as Ahlfors-Schwarz lemma.
The following states a version of the Ahlfors Schwarz lemma in the
form most suitable for us.

\begin{theorem}
Let $X$ be a Riemannian surface with hermitian metric $h$.
Assume that there is a constant $C>0$ such that the Gaussian
curvature of $h$ is bounded from above by $-C$.

Then the inequality
\[
||Df_p||\le \frac 1 C
\]
holds for every holomorphic map $f:\Delta\to X$ and every $p\in\Delta$,
where $||\ ||$ denotes the operator norm with respect to the hermitian
metric $h$ on $X$ and the Poincar\'e metric on the unit disc
$\Delta$.
\end{theorem}

For a proof see e.g.~Kobayashi. 

\section{Riemann surfaces of finite type}
\begin{definition}
A Riemann surface is called ``of finite type'' if its fundamental
group is finitely generated.
\end{definition}

Evidently every complex algebraic curve is a Riemann surface of finite
type. But there are also non-algebraic Riemann surfaces of finite type,
for example the unit disc and the annuli $A(r,1)=\{z\in\C:r<|z|<1\}$
($r\in[0,1[$).

The following well-known result is a key tool.
\begin{theorem}
Every Riemann surface $X$ of finite type admits 
an open embedding $X\hookrightarrow \hat X$ 
into a compact Riemann surface $\hat X$ such that every connected component of the complement
of $X$ in $\hat X$ is either a point
or isomorphic to a closed disc.
\end{theorem}

Sketch of proof:

Let $\rho$ be a subharmonic Morse function on $X$. Because $\pi_1$
is finitely generated, there are only finitely many critical points.
Therefore $X$ is diffeomorphic to $X_c=\{x\in X:\rho < c\}$ for some
$c>0$ and $X\setminus \bar X_c$ is homotopic to $\partial X_c$. Now
$\partial X_c$ is a compact real one-dimensional manifold. It follows
that each connected component of $X\setminus\bar X_c$ has an infinite
cyclic group as fundamental group. Using the uniformization theorem,
it follows that each connected component of $X\setminus \bar X_c$
is biholomorphic to $A(r,1)$ for some $1>r>0$.
Using the natural embedding $A(r,1)\hookrightarrow\Delta$
we obtain the compactification $X\subset\bar X$.

We introduce some notion.
\begin{definition}\label{std-cpt}
Let $X$ be a Riemann surface. A ``standard compactification'' of $X$
is an open holomorphic embedding $i:X\to Y$ into a compact
Riemann surface $Y$ 
such that for every connected component $W$ of $Y\setminus i(X)$ there
exists an open neighborhood $V$ of $W$ in $Y$, a biholomorphic
map $\phi:V\to\Delta$ and a real number $0\le r<1$ such that
$\phi(W)=\phi(V\setminus i(X))=\{z:|z|\le r\}$.
\end{definition}

The above result guarantees that every Riemann surface of finite
type admits a {\em standard compactification}.

\section{Definition of the hyperbolic length spectrum}

We recall the definition of the (infinitesimal)
Kobayashi--Royden pseudometric $F_X$: For a complex space $X$
one defines $(F_X)_p:T_pX\to\R$ as
\[
(F_X)_p(v)=\inf\{ |w| : \phi_*w=v,\ \exists\phi:(\Delta,0)\to(X,p)
\text{ holo.}\}
\]

The uniformization theorem implies that for a 
Riemann surface $X$ either
$F_X$ vanishes identically or $F_X$ is a complete K\"ahler metric
of constant negative curvature.

Then one can define the hyperbolic length $L(\gamma)$
of a differentiable path
$\gamma:S^1=\R/\Z\to X$ as
\[
L(\gamma) = \int_0^1 F_X(\gamma'(t))dt.
\]

Given a manifold $X$ (here always a Riemann surface), a
{\em simple path} is  injective and immersive 
differentiable map $\gamma:S^1\to X$.

Given a hyperbolic  Riemann surface $X$, 
let $\Gamma_X$ be the set of free homotopy classes
of  simple paths. This can be regarded as a subset of
the quotient space of the fundamental group $\pi_1(X)$
by the equivalence relation given by inner automorphisms
(i.e.~conjugation) of $\pi_1(X)$.

For every $\gamma\in\Gamma_X$, 
we define a ``stable hyperbolic length'' $\Lambda(\gamma)$
as the infimum of the hyperbolic
length of all simple paths representing $\gamma$.
Let
\begin{equation}
\Sigma_X=\{ \Lambda(\gamma):\gamma\in\Gamma_X\}\label{spectrum}
\end{equation}

$\Sigma_X$ is a countable subset of $\R^+_0$.

Obviously, $\Sigma_X$ is an invariant of the complex structure
on $X$. We will use this invariant to distinguish non-equivalent
complex structure.
For this purpose we will show that $\Sigma_X$ is non-trivial
(i.e.~$\Sigma_X\ne\{0\}$) for almost all hyperbolic
Riemann surfaces. Furthermore we will show that there is always
a family of complex structures with changing $\Sigma_X$ as
soon as $\Sigma_X\ne\{0\}$.

More precisely:

{\em  $\Sigma_X\ne\{0\}$ unless $X$ is isomorphic to 
$\Delta=\{
z:|z|<1\}$, 
$\Delta\setminus\{0\}$ or $\P_1\setminus\{0,1,\infty\}$.}

\section{Non-triviality of the length spectrum}

\begin{lemma}\label{subdelta}
Let $E\subset\Delta$ be a compact subset and let $X=\Delta\setminus E$.

If $E$ contains at least two points, there exists a simple path on $X$
which can not be deformed to a simple path of arbitrary small length.
\end{lemma}

\begin{proof}
Let $p,q\in E$ with $p\ne q$. Applying a suitable automorphism of the
unit disc, we may assume $p=0$. Then every simple path in $X$ surrounding
both $p$ and $q$ has a euclidean length of at least $2|q|$.
Now the euclidean distance is an lower bound for the Poincar\'e 
distance on the unit disc $\Delta$ which in turn is an lower bound
for the hyperbolic distance on $X\subset\Delta$.
Therefore every simple curve in $X$ surrounding all of $E$ has hyperbolic
length at least $2|q|$.
\end{proof}

%The aim of this section is the following:
%
%\begin{theorem}
%Let $X$ be a Riemann surface with $b_1(X)\ge 4$.
%
%Then there exists a simple curve $\gamma:S^1\to X$ and a number $c>0$ such that the
%hyperbolic length of every curve homotopic to $\gamma$ is at least $c$.
%\end{theorem}

\begin{lemma}\label{rc}
Let $X$ be a relatively compact 
smoothly bounded domain in a hyperbolic Riemann surface $Y$.

Then there exists a number $c>0$ such that every closed curve of hyperbolic
length less than $c$ in $X$ is homotopic to a constant map.
\end{lemma}
\begin{proof}
Because $X$ is relatively compact in $Y$, we may choose a number $r>0$
such that for all $\rho<r$ and all $x\in X$ the set
$\{y\in Y:d_Y(x,y)<\rho\}$ is homeomorphic to an open ball.
(Remember, that the infinitesimal Kobayashi pseudometric is actually
a Riemannian metric for $\dim Y=1$. Hence we may use the exponential
map for this argument.)

Because $X$ is furthermore smoothly bounded, there is a number $s>0$
such  that $W=\{y\in:d_Y(y,X)<s$ is homotopically equivalent to $X$.

Together, this two assertions imply that every closed curve
whose length with respect to the hyperbolic metric of $Y$
is less than $\min\{s,r\}$ must be homotopic to zero.

Since the inclusion map $X\to Y$ is distance decreasing for the
respective Kobayashi metrics, the desired assertion follows.
\end{proof}

\begin{lemma}\label{retract-to-bdry}
Let  $X$ be a hyperbolic  Riemann surface and 
let $\Omega$ be a relatively compact open
domain in $X$ with smooth boundary.
Assume $\Omega\ne X$. 
Let $\gamma:S^1\to X$ be a simple path  which can be deformed (in $X$)
to a curve of arbitrarily small hyperbolic length. 
Then $\gamma$ can be deformed to a curve $\tilde\gamma$
with $\tilde\gamma(S^1)\subset X\setminus\overline{\Omega}$.
\end{lemma}

\begin{proof}
The assertion is trivially true if $\gamma$ is homotopic to the constant map.
Hence we may and do assume that $ \gamma$ is not homotopic to the constant map.
For $\epsilon>0$ let $\Omega_\epsilon=\{x\in \Omega: d_X(x,\partial \Omega)>\epsilon\}$.
Since $\Omega$ is relatively compact with smooth boundary, 
 for every sufficiently small $\epsilon$ there exists a diffeomorphism $\phi$ of $X$
homotopic to the identity map with $\phi(\Omega)=\Omega_\epsilon$.
Next we choose a positive number $\delta>0$, such that $B_\delta(p)=\{x\in X:d_X(x,p)<\delta\}$
is contained in a contractible subset of $\Omega$ for every $p\in\Omega_{\frac\epsilon2}$. This is possible, because 
$\Omega_{\frac\epsilon2}$ is relatively compact.
If $\gamma:S^1\to X$ is a curve of hyperbolic length $\delta$, our assumptions
imply that $\gamma$ is homotopic to a constant map or $\gamma(S^1)\not\subset \Omega_{\frac\epsilon2}$.
If its hyperbolic length is smaller than $\epsilon/2$, the condition
$\gamma(S^1)\not\subset \Omega_{\frac\epsilon2}$ implies $\gamma(S^1)\cap \Omega_\epsilon=\{\}$.
Thus every curve which can be deformed to a curve of arbitrarily small hyperbolic length is
homotopic to a curve $\gamma$ whose image is contained in $X\setminus\Omega_\epsilon$.
But this in turn implies that $\tilde\gamma=\phi^{-1}\circ\gamma$ is homotopic to $\gamma$ with
$\tilde\gamma(S^1)\subset X\setminus\Omega$.
\end{proof}

\begin{lemma}\label{fin-q-inj}
Let $X$ be a Riemann surface with a finite subset $E$ and let $\rho:X\to X_1$ be the quotient space obtained
by collapsing $E$ to a one point.

Then:
\begin{enumerate}
\item
The natural group homomorphism 
$\rho_*:\pi_1(X)\to\pi_1(X_1)$ is injective.
\item
If $\gamma:[0,1]\to X$ is a continuous path in $X$ such that $\gamma(0)$ and
$\gamma(1)$ are {\em different} points in $E$,
then
$\rho\circ\gamma$ defines a non-trivial element in $\pi_1(X_1)$.
\end{enumerate}
\end{lemma}
\begin{proof}
We prove the first statement for possibly singular Riemann surfaces. Then by induction it suffices
to consider the case where $E$ consists of two elements, say $E=\{a,b\}$. Define $Y_0=X\times \Z$
and let $Y$ denote the quotient obtained by identifying $(b,n)$ with $(a,n+1)$ for each $n\in\N$.
Now we have a free and properly discontinuous $\Z$-action on $Y$ induced by $m:(x,n)\mapsto (x,n+m)$.
The natural projection $\pi:Y\to X_1$ is an unramified Galois covering for this $\Z$-action. Therefore
$\pi_1(Y)\to\pi_1(X_1)$ is injective. We can embedd $X$ into $Y$ via $i:x\mapsto[(x,0)]$. Then the
natural projection $\rho$ from $X$ to $X_1$ is given as $\rho=\pi\circ i$. Now $\tau:Y\to X$ induced by the
map $\tau_0\to X$ defined as
\[
\tau_0: (x,n)\mapsto\begin{cases} a & \text{ if $n<0$ }\\
x & \text{ if $n=0$ }\\
b & \text{ if $n>0$} \\
\end{cases}
\]
has the property $id_X=\tau\circ i$. It follows that $i_*:\pi_1(X)\to\pi_1(Y)$ is injective.
Hence $\rho_*=(\pi\circ i)_*:\pi_1(X)\to\pi_1(X_1)$ is injective as well.
\end{proof}

\begin{proposition}\label{genus-1}
Let $\Omega\subset X$ be a smoothly bounded relatively compact domain in a Riemann surface
and let $\Omega\subset\overline{\Omega}$ be a standard compactification (in the sense of definition~\ref{std-cpt}).

Assume that the genus of $\overline{\Omega}$ is at least one.

Then there exists a simple curve $\gamma:S^1\to X$ and a number $c>0$ such that the
hyperbolic length of every curve homotopic to $\gamma$ is at least $c$.
\end{proposition}
\begin{proof}
By assumption  $\pi_1(\overline{\Omega})\ne\{e\}$.
Since $\overline{\Omega}\setminus\Omega$ can be contracted to a finite set, 
it is clear that
every element in $\gamma_0\in\pi_1(\overline{\Omega})\setminus\{e\}$
can be represented by a curve
with image inside $\Omega$ and that there exists a simple curve $\gamma:S^1\to\Omega$ which is not
homotopic to a constant map as a map from $S^1$ to $\overline{\Omega}$.

Now let $\Omega'$ be the one-point-compactification of $\Omega$.
Due to lemma~\ref{fin-q-inj} we know that $\gamma$ gives us a non-trivial element
in $\pi_1(\Omega')$. Since we have a natural continuous map from $X$ to $\Omega'=\Omega\cup\{\infty\}$
given by
\[
x \mapsto \begin{cases} x & \text{ if $x\in\Omega$ }\\
\infty & \text{ if $x\not\in\Omega$ }
\\
\end{cases},
\]
it follows that $\gamma$ can not be deformed in $X$ to a curve whose image is contained
in $X\setminus\bar\Omega$.
Now the assertion follows from lemma~\ref{retract-to-bdry}.
\end{proof}

\section{The case of genus $0$}

\begin{lemma}\label{lem-pi-0}
Let $\Omega_n$ be an increasing sequence of relatively compact
smoothly bounded domains in a Riemann surface $X$ such that
$\cup_n \Omega_n=X$.

Assume that every $\Omega_n$ admits a standard compactification
$\hat\Omega_n$
of genus $g=0$.

Then for every $n$ the inclusion $\partial\Omega_n\to X\setminus \Omega_n$
induces a bijective correspondance 
$\pi_0(\partial\Omega_n)\simeq \pi_0(X\setminus \Omega_n)$
between the sets of connected components.
\end{lemma}
\begin{proof}
We have to show: {\em If $p,q$ are points in two distinct connected
components of $\partial\Omega_n$, then they can not be connected
by a path inside $X\setminus\Omega_n$}. We assume the contrary.
Then $p,q$ can be connected by a path inside $\Omega_m\setminus\Omega_n$
for some sufficiently large $m>n$.
We concatenate this path with a path connecting $q$ and $p$
inside $\overline{\Omega_n}$ and obtain an element
$\gamma\in\pi_1(\Omega_m)$. Due to lemma~\ref{fin-q-inj} $(ii)$
$\gamma$ projects onto a non-trivial element in 
$\pi_1(\Omega_n')$ where $\Omega_n'$ denotes the one-point-%
compactification of $\Omega_n$.
Since the identity map of $\Omega_n$ extends in an obvious way
to a continuous map from $\hat\Omega_m$ to $\Omega_n'=\Omega_n\cup
\{\infty\}$ by mapping every point in $\hat\Omega_m\setminus\Omega_n$
to $\infty$, it follows that $\gamma$ defines a closed curve in
$\hat\Omega_m$ which is not homotopic to a constant map.
But this contradicts our assumption that each $\hat\Omega_m$
has genus zero.
\end{proof}

\begin{corollary}\label{deform-to-bdry}
Under the above assumptions, let $\gamma$ be a closed curve in $\Omega_n$
which can be deformed into a closed curve in $X\setminus\Omega_n$.

Then $\gamma$ can be deformed to a closed curve in $\partial\Omega_n$.
\end{corollary}
\begin{proof}
Let $W$ be a connected component of $X\setminus\Omega_n$ such that
$\gamma$ can be deformed to a closed curve inside $W$.
The above proposition implies that $W\cap\bar\Omega_n$
is a connected component of $\partial\Omega_n$.
Therefore the Seifert-van-Kampen theorem may be applied to
$\bar\Omega_n\cup W$. It follows that $\gamma$ can be deformed
to a closed curve in $W\cap\bar\Omega_n\subset\partial\Omega_n$.
\end{proof}
\begin{corollary}\label{rank-2}
Under the assumptions of the proposition,
assume that every simple path in $X$ may be deformed to 
a closed curve of arbitrarily small hyperbolic length.

Then for every $n$, the natural
linear map $i_*:H_1(\Omega_n,\R)\to H_1(X,\R)$
induced by the inclusion map $\Omega_n\subset X$
has rank at most $2$.
\end{corollary}

\begin{proof}
Let $r$ be the cardinality of the set of
connected components of $\partial\Omega_n$.
Then $\Omega_n$ is homeomorphic to
$\P_1\setminus\{1,2,\ldots,r\}$.
Let $\gamma_s$ for $s\in I=\{1,\ldots,r\}$ denote a small cycle around $s$.
Then $H_1(\Omega_n,\Z)$ is the $\Z$-module generated by the $\gamma_s$
subject to the relation $\sum_s\gamma_s=0$.
Since every permutation of $\{1,2,\ldots,r\}$ extends to a
homeomorphism of $\P_1$, it is clearly that for every
subset $B\subset\{1,\ldots,r\}$ the homology class
$\sum_{s\in B}[\gamma_s]\in H_1(\Omega_n,\Z)$
can be realized by a simple curve in $\Omega_n$.

Next, consider injections $\Omega_n\subset\Omega_m$ for $m>n$.
To each $s\in I$ there corresponds a connected component
$a_s\in \pi_0(\partial \Omega_n)\simeq\pi_0(X\setminus\Omega_n)$.
(for the latter equivalence see lemma~\ref{lem-pi-0}).

Given an index $s\in I$, the cycle $\gamma_s$ maps to a non-zero
homology class in $H_1(\Omega_m,\Z)$ if and only if
there is a connected component of $\partial\Omega_m$ included
in the connected component of $X\setminus\Omega_n$ corresponding to $s$.

Now let $I_0$ denote the set of all $s\in I$ such that for all
$m>n$ there is a connected component of $X\setminus\Omega_m$
included in the connected component of $X\setminus\Omega_n$
corresponding to $s$.

Then the induced map $i_*:H_1(\Omega_n;\Z)\to H_1(X,;\Z)$
maps $\gamma_s$ to zero if and only if $s\not\in I_0$
and moreover for the elements $s\in I_0$ the homology classes of
$\gamma_s$ in $H_1(X,\Z)$ are non-zero and subject only to the
relation $\sum_{s\in I}[\gamma_s]=0$.
As a consequence, if $I_0$ contains at least four elements,
say $1,2,3,4$, we may take two of them, say $1,2$ and there exists
a simple path $\gamma\in \Omega_n$ such that the homology class
of $[\gamma]$ equals $[\gamma_1]+[\gamma_2]$.

Thus in this case there exists a simple path in $\Omega_n$
which in $X$ is not homologous to any multiple of one of 
the $\gamma_i$ ($i\in I$), i.e., not homologous to any
closed curve inside
$\partial\Omega_n$.

Then by corollary~\ref{deform-to-bdry}, 
we found a simple path in $X$ which can not be deformed
to a curve of arbitrarily small hyperbolic length in $X$.

\end{proof}

\begin{theorem}\label{8-2}
Let $X$ be a hyperbolic
Riemann surface. Assume that at least one of the following
conditions is fulfilled:
\begin{enumerate}
\item
There exists a smoothly bounded relative compact domain $\Omega$
whose standard compactification has genus at least $1$.
\item
$b_1=\dim H_1(X,\R)\ge 3$.
\item
There exists an open embedding of $X$ as a relatively compact domain
into an other hyperbolic Riemann
surface .
\end{enumerate}

Then there exists a simple path which can not be deformed to
a path of arbitrarily small hyperbolic length.
\end{theorem}
\begin{proof}
\begin{enumerate}
\item
 see proposition~\ref{genus-1}.
\item
Let $\Omega_n$ be an increasing sequence of relatively compact smoothly
bounded domains in $X$ exhausting all of $X$.
Then $H_1(X,\R)=\lim H^1(\Omega_n,\R)$. Hence $b_1(X)\ge 3$
implies that there exists a number $n$ such that
$i_*:H_1(\Omega_n,\R)\to H_1(X,\R)$ has at least rank three.
Furthermore, thanks to $(i)$ we may now assume that all the
$\Omega_n$ have a standard compactification of genus $0$.
Thus the assertion follows from corollary~\ref{rank-2}.
\item
This is lemma~\ref{rc}.
\end{enumerate}
\end{proof}
\begin{corollary}\label{arb-small}
Let $X$ be a hyperbolic Riemann surface. Assume that every simple
curve in $X$ can be deformed to a curve of arbitrarily small
hyperbolic length.

Then $X$ is biholomorphic to one of the following:
\begin{enumerate}
\item
the unit disc $\Delta=\{z\in\C:|z|<1\}$,
\item
the punctured unit disc $\Delta^*=\{z\in\C:0<|z|<1\}$,
 \item
$\P_1\setminus\{0,1,\infty\}$,
\end{enumerate}
\end{corollary}
\begin{proof}
By the above theorem~\ref{8-2} $(ii)$ we have $b_1(X)\le 2<+\infty$. 
Hence $X$ is of finite type and admits a ``standard compactification''
$X\hookrightarrow\hat X$. Furthermore $\hat X \simeq\P_1$ 
(Theorem~\ref{8-2} $(i)$).
By the definition of a standard compactification each connected
component of $\hat X\setminus X$ is a point or a closed disc.
There are at most three connected components of $\hat X\setminus X$
because of $b_1(X)\le 2$.
If there are only isolated points and no closed discs, then there
are three points in $\hat X\setminus X$ (otherwise $X$ would not be
hyperbolic) and $X\simeq\P_1\setminus\{0,1,\infty\}$
(because $PSL_2(\C)$ acts triply transitive on $\P_1$).
This leaves the case where one of the connected components
is a closed disc. Let $B$ be this component. Then $\hat X\setminus B$
is isomorphic to the unit disc. Thus $X$ is isomorphic to the complement
of a compact set $K$ (possibly empty) in the unit disc.
Due to lemma~\ref{subdelta} the set $K$ contains at most one point.
Therefore either $X\simeq\Delta$ or $X\simeq\Delta^*$.
\end{proof}

\begin{remark}
An explicit calculation shows that the hyperbolic length
of the curve $[0,2\pi]\ni t\mapsto re^{it}$ in $\Delta^*$
equals $\frac{2\pi}{-\log r}$.

Using this fact one verifies easily that for each of the three Riemann surface
mentioned above ($\Delta$, $\Delta^*$ and $\P_1\setminus\{(0,1,\infty\}$) it is indeed
true that every simple path can be deformed to a path of arbitrarily small
hyperbolic length.
\end{remark}

\section{Shrinking the length}
We start with an elementary observation.
\begin{lemma}
Let $\log r>\alpha>2\pi$.
Then there is a holomorphic map from the unit disc $\Delta$ to the annulus
$A(1/r,r)=\{z:1/r<|z|<r\}$ mapping the real interval $[0,\frac{2\pi}{\alpha}]$ onto
$S^1=\{z:|z|=1\}$.
\end{lemma}
\begin{proof}
Take $z\mapsto \exp(i\alpha z)$ and note that $|z|<1$ implies $|Im(z)|<1$ which in 
combination with $\log r>\alpha$ implies $1/r<|\exp(i\alpha z)|<r$.
\end{proof}
\begin{corollary}\label{length-circle}
For $r>1$ let $\mu(r)$ denote the hyperbolic length of $S^1$ as a curve in the
Riemann surface $A(1/r,r)$.

Then $\lim_{r\to\infty}\mu(r)=0$.
\end{corollary}
\begin{proof}
Due to the lemma $\mu(t)$ is bounded from above by the hyperbolic length of
$[0,\frac{2\pi}{\alpha}]$. % Give explicit value ?
for $t>2\pi$.
This implies the assertion.
\end{proof}

\begin{lemma}\label{shrink}
Let $\gamma$ be a simple curve in a Riemann surface $(X,J)$. Let $\epsilon>0$.

Then there 
exists a smooth family of complex structures $J_t$ ($t\in [0,1]$)
 such that
\begin{enumerate}
\item
the hyperbolic length of $\gamma$ with respect to the complex structure $J_1$ is less than $\epsilon$.
\item
$J_0$ equals the given complex structure $J$.
\item
There is a compact subset $K$ of $X$ such that $(J_t)_x=(J_s)_x$ for all $x\not\in K$, $s,t\in[0,1]$, i.e., all the complex structures $J_t$ agree
outside $K$.
\end{enumerate}
\end{lemma}

\begin{proof}
We choose real constants $r>r'>1$.
Then we choose an open neighbourhood $V$ of $\gamma$ which admits a diffeomorphism $\phi$ to the annulus
$A(1/r,r)$ taking $\gamma$ to the unit circle $S^1=\{z:|z|=1\}$. 
Define $W=\phi^{-1}(A(1/r',r'))$.
Choose a smooth real function $\chi:X\to[0,1]$ such that 
\begin{enumerate}
\item
$\chi|_W\equiv 1$,
\item
$\chi$ is constant zero in some open neighbourhood of $X\setminus V$.
\end{enumerate}
Next we choose a hermitian metric $h$ on $X$ and a hermitian metric $\rho$ on $A(r,s)$.
We define $H_t=(1-t\chi)h+t\chi\phi^*\rho$. This is a Riemannian metric on $X$ which
determines a complex structure $J_t$ on $X$ with $J_0$ being the original complex
structure on $X$. By construction $(W,J_1)$ is biholomorphic to $A(1/r',r')$.
Since the choice of $r>r'>1$ was arbitrary, the value of $r'$ may be as large as desired.
Then the hyperbolic length of $\gamma$
with respect to $(W,J_1)$ becomes as small as desired
(corollary~\ref{length-circle}).
Since the injection of $(W,J_1)$ into $(X,J_1)$ is distance-decreasing,
the claimed assertion follows.
\end{proof}

\section{Some preparation}
\begin{lemma}
Let $M$ be a (real) differentiable manifold and let $f_n:M\to M$ be
a sequence of smooth self-maps $C^1$-converging to the identity map.

Let $K\subset K'$ be compact subsets of $M$. Assume that $K$ is connected
and contained in the interior of $K'$.

Then there exists a number $N$ such that the following
assertions hold for all $n\ge N$:
\begin{enumerate}
\item
The restricted map $f_n|_{K'}$ is injective.
\item
The image $f_n(K')$ contains $K$.
\end{enumerate}
\end{lemma}
\begin{proof}
Assume that there are arbitrary large numbers $n$ for which $f_n|_{K'}$
is not injective.
Then there exists sequences $p_n,q_n\in K$ such that (after passing
to a suitably chosen subsequence) $p_n\ne q_n$, but $f_n(p_n)=f_n(q_n)$.
We may assume that $p_n$ and $q_n$ are both convergent. Then
there is a point $p\in K'$ such that
\[
\lim p_n=\lim f_n(p_n)=\lim f_n(q_n)=\lim q_n.
\]
Locally, i.e., in a neighbourhood of $p$, we may embed everything in the
euclidean space. Then the above implies that there are vectors of unit length
$v_n$ and points $\xi_n$ on the segment between $p_n$ and $q_n$ such that
\[
D_{v_n}(f_n)_{\xi_n}= \lim_{t\to 0}\frac{ f_n(\xi_n+tv_n)-f_n(\xi_n)}{t}=0
\]
Taking the limit, we obtain that $\lim f_n$ has at $p$ a zero directional
derivative in some direction. This contradicts the assumption that
$f_n\to id_M$ in $C^1$-topology.

Next we deal with the second assertion. By enlarging $K$ (if necessary)
we may assume that $K$ has non-empty interior. Let $q\in int(K)$.
We fix some metric on $X$ defining the topology.
Choose $\epsilon>0$ such that
$\epsilon$ is smaller that the distance between $K$ and
$X\setminus K'$ and furthermore smaller than the distance between $q$
and $\partial K$.

Then we can choose a number $N$ such that for all $n\ge N$ we 
have 
\[
\sup_K d(x,f_n(x))<\epsilon\ \ \forall n\ge N, \forall x\in K'
\]
In addition, we may and do require that $(Df_n)_x$ is invertible
for all $x\in K'$ and $n\ge N$.

We fix a number $n\ge N$ and consider the set
$A=K\cap f_n(K')$. By construction $f_n(q)\in A$, hence $A$ is not empty.
The set $A$ is closed, because $K$ and $K'$ are compact.
On the other hand, if $p\in K'$ with $f_n(p)\in K$, then
$d(p,f_n(p))<\epsilon$ implies $d(p,K)<\epsilon$, which in turn
implies that $p$ is in the interior of $K'$. 
Now $f_n$ is locally a diffeomorphism, since $(Df_n)_x$ is invertible
for all $x\in K'$. Therefore, if $p\in K'$ with $f_n(p)\in K$,
then $f_n(p)$ admits an open  neighbourhood in $X$ which is also
contained in the image $f(K')$. As a consequence, the set
$A=K\cap f(K')$ is both closed and open in $K$. Since $A$ is non-empty
and $K$ is connected, it follows that $A=K$, i.e., it follows
that $K\subset f_n(K')$.
\end{proof}

\begin{corollary}\label{ex-inverse}
Under the above assumptions for $n\ge N$ there is a unique
inverse map $g_n:K\to K'$, i.e., a unique map $g_n:K\to K'$
with $f_n\circ g_n=id_K$.
\end{corollary}

\section{Continuity of the Kobayashi pseudodistance}

\begin{theorem}\label{F-cont}
Let $S$ be an orientable real surface with a smooth family of complex structures $J_t$.
Assume that there is a relatively compact subset $\Omega$ 
such that all the
complex structures $J_t$ agree outside $\Omega$.

Then the map $F:TS\times I\to\R^+_0$ given by the Kobayashi-Royden
pseudometric on $TS$ with respect to $J_t$ ($t\in I$) is continuous
on $TS\times I$.
\end{theorem}

\begin{remark}
It is important that we deform the complex structure only inside some
fixed relatively compact subset. Without this assumption the statement
is not true. For example, we have seen that there is a family of complex
structure $J_t$ on $D=\{z:|z|<1\}$ such that $(D,J_0)$ is biholomorphic
to the unit disc while $(D,J_t)\simeq\C$ for every $t\ne 0$.
Evidently for this family $F$ is not continuous, since it vanishes for
$t\ne 0$ and is non-zero for $t=0$.
\end{remark}

\begin{proof}
We will need the result only for the case where $X$ is hyperbolic.
However, it is easy to see that the statement holds if $X$ is not hyperbolic:
$X$ is not hyperbolic precisely if and only if one of the following
conditions are fulfilled:
\begin{enumerate}
\item
$X$ is compact and $b_1(X)\le 2$.
\item
There exists a complex analytic compactification
$X\hookrightarrow\bar X$ such that $\bar X$ is compact
and simply-connected and $\bar X\setminus X$ contains at most
two points.
\end{enumerate}
In this formulation it is clear that the property of not being
hyperbolic can not be changed by modifying the complex
structure only inside some fixed compact set.
Therefore $d_{(X,J_t)}$ vanish for all $t$
if $d_{(X,J_0)}$ vanishes. In particular, in this case
the Kobayashi pseudodistance depends continuously on $t$
(because it is constantly zero).

Thus
we may from now on assume that $(X,J_0)$ is hyperbolic.

\end{proof}
\begin{proposition}\label{uniform-semi}
Let $X$ be an orientable real surface with a smooth family of complex structures $J_t$.
Assume that there is a relatively compact subset $\Omega$ 
such that all the
complex structures $J_t$ agree outside $\Omega$.

Assume that $(X,J_0)$ is hyperbolic.

Then for every $\varepsilon>0$ there exists a $\delta>0$ such that for
all $x\in X$,
 $v\in T_xX$ and $t\in[0,\delta]$ the inequality
\[
F_{(X,J_t)}(v)\le (1+\varepsilon) F_{(X,J_0)}(v)
\]
holds.
\end{proposition}
In words: The Kobayashi pseudometric is uniformly upper semi-continuous
with respect to $t$.

\begin{proof}
Let $h_t$ be a family of hermitian metrics for $(X,J_t)$ varying smoothly
with $t$. 
Let $g$ be the complete hermitian
metric of constant Gaussian curvature $-1$ on $(X,J_0)$
(which is unique and exists because $(X,J_0)$ is hyperbolic).
Then there is a positive function $ \rho:X\to\R$ such that $h_0=\rho\cdot g$.
We define $g_t=\rho\cdot h_t$.
Let $G_t$ denote the Gaussian curvature for $g_t$. Then $G(t,x)=G_t(x)$ is continuous
and equals $-1$ on 
$\left(\{0\}\times X\right )\cup
\left([0,1]\times(X\setminus K)\right)$.
Since $K$ is compact, we can find a number $c>0$ and a constant $1>C>0$ such that
$G_t\le -C$ for all $t\le c$ and at every point of $X$.
Furthermore, once again by compactness of $K$, there is an other constant $C'>0$
such that $g_0=C'g$.
Using Ahlfors-Schwarz lemma it now follows that 
\[
||\phi_*v||_{g}\le \frac{C'}{C}||v||
\]
for every $v\in T_0\Delta$ and every holomorphic map $\phi:\Delta\to(X,J_t)$
($t\in[0,c]$).
The family of all such maps $\phi$ is therefore equicontinuous.
Due to the theorem of Arzela Ascoli
they form a normal family.
Now fix $t$, $p\in X$ and let $(v_n,t_n,x_n)$ be a sequence
 such that $\lim x_n=x$, $\lim t_n=t$ and $\lim v_n=v\in T_xX$.
Let $\alpha=\limsup_n (F_{(X,J_{t_n})})_{x_n}(v_n)$. By the definition
of the infinitesimal Kobayashi-Royden
there is a sequence of holomorphic maps 
\[
\phi_n:\Delta\to(X,J_{t_n}),\quad,\quad \phi_n(0)=x_n, 
(\phi_n)_*\alpha_n\frac{\partial}{\partial z}=v_n
\]
with 
$\lim \alpha_n=\alpha$. Due to the normal family property the sequence $\phi_n$ admits a convergent
subsequence, i.e., there is a holomorphic map
\[
\phi:\Delta\to(X,J_{t}),\quad,\quad \phi(0)=x, (\phi)_*\alpha\frac{\partial}{\partial z}=v
\]
It follows that
\[
F_{(X,J)}(v)\le\alpha \liminf_n (F_{(X,J_{t_n})})_{x_n}(v_n)
\]
This establishes lower semicontinuity.
Upper semicontinuity follows from proposition \ref{usc} below.
\end{proof}
\begin{proposition}
Let $(J_t)_{0\le t\le 1}$ 
be a smooth family of complex structures on a 
orientable real surface $S$.
Let $K$ be  a compact subset of $S$, $K\ne S$.
Then there is a number $c>0$ such that
for all $t\in[0,c]$
there exists a holomorphic injective map $\phi_t:K^\circ\to (S,J_t)$
such that $\phi_t$ converges uniformly to the identity map $id_S$
for $t\to 0$.
\end{proposition}
\begin{proposition}\label{usc}
Let $(J_t)_{0\le t\le 1}$ 
be a smooth family of complex structures on a 
orientable real surface $S$.

The Kobayashi-Royden pseudodistance is upper-semicontinuous as
a function on $TS\times[0,1]$.
\end{proposition}

\begin{proof}
Let $v\in T_pS$, $p\in S$. Assume $F_{(S,J_0)}(v)=c\in\R$.
Fix $\epsilon>0$.
Then there is a holomorphic map
$f:(\Delta,0)\to (X,p)$ with 
$f_*\ddz=\lambda v$, $|\lambda|>c-\epsilon$.
Define $F(z)=f((1+\epsilon)z)$. Then
$F_*\ddz=(1+\epsilon)\lambda v$. By construction
the image $F(\Delta)$ is now contained in the compact set
\[
K_0=f\left(\left\{z:|z|\le\frac{1}{1+\epsilon}\right\}\right).
\]

The family of complex structures $(J_t)$ endow the product 
$S\times [0,1]$ with the structure of
a $CR$-manifold which is foliated by complex leaves, 
namely the $(S,t)$. Hence we have
a Levi-flat real three-dimensional $CR$-hypersurface. 
Such a $CR$-manifold can be embedded
into a complex manifold $Y$
(\cite{HLN}). Now $S_0=(S,0)$ is a closed complex Stein submanifold of $Y$.
Let $\Omega$ be a Stein open neighbourhood of $S_0$ in $Y$ (which exists due to \cite{Siu}).
After shrinking $\Omega$ if necessary there exists a holomorphic retraction $\rho:\Omega\to S_0$
(see e.g.~\cite{DLS}, lemma~2.1).

Let $K$ be a compact subset of $S$ containing $K_0$ in its interior.

For every $t\in[0,1]$  we consider $K_t=\{(x,t):x\in K\}\subset Y$.
There is a bound $r>0$ such that $K_t\subset\Omega$ for all
$t<r$ due to compactness of $K$ and openness of $\Omega$.

Let $\zeta_t$ denote the map
given by
\[
\zeta_t: x\mapsto (x,t)\in Y.
\]
Now $\rho\circ \zeta_t$ converges uniformly on $K$ to the identity map
in $C^1$-topology
for $t\to 0$. Using corollary~\ref{ex-inverse} we may deduce that
for $t$ sufficiently close to $0$ there is a map
$g_t:K_0\to K$ with $\rho\circ\zeta_t\circ g_t=id_{K_0}$.
Now we define (for sufficiently small $t$) a map 
$F_t:\Delta\to S_t$ via
\[
F_t=\zeta_t\circ g_t\circ F.
\]
We obtained a family of holomorphic maps $F_t:\Delta\to (S,J_t)$
with $\lim F_t=F$.
This yields the desired semicontinuity.
\end{proof}

\begin{proposition}\label{lambda-cont}
Let $X$ be an orientable surface and $J_t$ a smooth family of
complex structures which agree outside a compact subset $K\subset X$.
Let $\gamma\in\pi_1(S)$ and for each $t$ define $\lambda(t)$ as
the infimum of the hyperbolic length 
(with respect to the complex structure $J_t$) of closed simple curves
(freely) homotopic to $\gamma$.

Then $t\mapsto\lambda(t)$ is continuous.
\end{proposition}
\begin{proof}
Let $\xi$ be a fixed closed curve freely homotopic to $\gamma$.

By theorem~\ref{F-cont} we know that its hyperbolic length
\[
L_{(X,J_t)}(\xi)=\int F_{(X,J_t)}(\xi')ds
\]
is continuous in $t$.

The infimum of a family of continuous functions
is always upper-semicontinuous.

Thus it suffices to show that $\lambda(t)$ is lower-semicontinuous.
Due to proposition~\ref{uniform-semi}
we know that for every simple closed curve $\zeta$,
every $t_0$ and every $\varepsilon>0$
there exists a real number $\delta>0$ such that
\[
L_{t}(\zeta) > L_{t_0}(\zeta)/(1+ \varepsilon)
\]
for all $t$ with $|t-t_0|<\delta$.
This implies immediately
\[
\lambda(t) \ge \lambda(t_0)/(1+ \varepsilon)
\]
for all $t$ with $|t-t_0|<\delta$.
Hence the desired lower semicontinuity.
\end{proof}

\section{Proof of the main theorem}
Here we prove the Main Theorem.
\begin{proof}
We distinguish three cases:
\begin{enumerate}
\item
$X$ is not hyperbolic.
\item
$X$ is hyperbolic, but the hyperbolic length spectrum is trivial.
\item
$X$ is hyperbolic and the hyperbolic length spectrum is not trivial.
\end{enumerate}

Case $(1)$:

If $X$ is compact and non-hyperbolic, then $X$ is biholomorphic to
$\P_1$ or an elliptic curve and the result follows from the
classical theory of moduli spaces of compact Riemann surfaces.

If $X$ is non-compact and non-hyperbolic, then $X\simeq\C$
or $X\simeq\C^*$ und we refer to proposition~\ref{case-special}.

Case $(2)$:

In this case we know from corollary~\ref{arb-small}
that $X\simeq\Delta$, $X\simeq\Delta^*$ or
$X\simeq\P_1\setminus\{0,1,\infty\}$ and the statement follows
from proposition~\ref{case-special}.

Case $(3)$:

In this case $X$ is hyperbolic and there exists a simple
closed curve $\gamma$ and a constant $c>0$ such that every
simple closed curve homotopic to $\gamma$ has hyperbolic length
at least $c$.
Fix such a curve $\gamma$ and let $\lambda(\gamma)=c$ denote its
``stable hyperbolic length'', i.e., the infimum of the hyperbolic
length $L(\tilde\gamma)$ where the infimum is taken over all simple
closed curves homotopic to $\gamma$.

Lemma~\ref{shrink} implies that there is a compact subset $K\subset X$
and a smooth family of complex structures $J_t$ on $X$ such that all
these complex structures agree outside $K$, and such that the
hyperbolic length $L_{(X,J_1)}(\gamma)$ of $\gamma$ with respect
to the complex structure $J_1$ fulfills the inequality
$L_{(X,J_1)}(\gamma)<c$.
For each $t$ we define $\lambda(t)$ as
the infimum of the hyperbolic length 
(with respect to the complex structure $J_t$) of closed simple curves
(freely) homotopic to $\gamma$.
Now $t\mapsto \lambda (t)$ is continuous due to 
proposition~\ref{lambda-cont} and furthermore non-constant
by construction ($\lambda(0)=c>\lambda(1)$).
We recall the definition of $\Sigma_{(X,J_t)}$ as in \ref{spectrum}.
We observe that $\Sigma_t=\Sigma_{(X,J_t)}$ is a countable subset
of $\R$ for every $t$. Since $\lambda$ is continuous and non-constant,
there exists a parameter $s$ such that $\lambda(s)\not\in\Sigma_0$.
Then $\Sigma_0\ne\Sigma_1$ and therefore $(X,J_0)$ is not biholomorphic
to $(X,J_s)$.
\end{proof}

\section{Deformations of discrete subgroups of $PSL_2(\R)$}

Our result on deformations of complex structures on Riemann surfaces
can be translated into a result on deforming discrete subgroups
of $PSL_2(\R)$.

\begin{theorem}
Let $F$ be  free group (possibly not finitely generated)
and let $\rho_0:F\to G=PSL_2(\R)$ be a group
homomorphism which embedds $F$ into $G$ as a discrete subgroup.

Then there exists a continuous family of group homomorphisms 
$\rho_t:F\to G$ ($t\in[0,1]$) such that each $\rho_t$ embedds $F$
as a discrete subgroup in $G$ and such that $\rho_0$ is not
conjugate to $\rho_1$. (I.e.~there is no $g\in G$ such that
$\rho_1(\gamma)=g\cdot \rho_0(\gamma)\cdot g^{-1}$.)
\end{theorem}

\begin{lemma}
Let $S$ be a hyperbolic Riemann surface, $I=[0,1]$, $M=S\times I$.

Let $A$ be a closed subset of $M$. Assume that $A$ equals the closure
of its interior.

For $t\in I$, let $X_t=\{p\in S:(p,t)\not\in A\}$.

Then the Kobayashi-pseudodistance on $X_t$ is continuous in $t$.
\end{lemma}

\begin{proof}
Let $t\in I$ and $p\in X_t$ and let $(p_n,t_n)\in M\setminus A$ with
$\lim (p_n,t_n)=(p,t)$.
Let $\phi_n:(\Delta,0)\to (X_{t_n},p_n)$ be a sequence of holomorphic maps.
Since $S$ is hyperbolic, there is a subsequence converging to a 
holomorphic map $\phi:(\Delta,0)\to (S,p)$.
By the open mapping theorem the image $\phi(\Delta)$ is open in $S$.
On the other hand $\phi(\Delta)\times\{t\}$ can not intersect the interior
of $A$, since $\phi=\lim\phi_n$. In combination with our assumption on $A$
it follows that $\phi(\Delta)\subset X_t$.

Conversely let $\psi:(\Delta,t)\to (X_t,p)$ be a holomorphic map.
For every $\epsilon>0$ the set $\psi(\bar\Delta_{1-\epsilon}$ is compact.
Since $M\setminus A$ is open, it follows that for every $\epsilon>0$
there exists a $\delta>0$ such that the map
$\psi_\epsilon:\Delta\to S$ defined by
\[
z\mapsto \psi\left( (1-\epsilon)z\right)
\]
has its image contained in $X_s$ for all $s$ with $|s-t|<\delta$.

These two considerations imply that the Kobayashi pseudodistance is
both upper- and lower semicontinuous, hence continuous in $t$.
\end{proof} 

\begin{corollary}
Let $X$ be biholomorphic to $\Delta^*$ or $\P_1\setminus\{0,1,\infty\}$.

Then there exists a deformation of the complex structure such that the Kobayashi
pseudodistance varies continuously.
\end{corollary}

\begin{proof}
Note that $\P_1\setminus\{0,1,\infty\}\simeq\C^*\setminus\{2\}$.
We use the lemma with 
$S=\Delta$ resp.~$S=\C^*\setminus\{2\}$ and
$A=\{(z,t):|z|\le t\}$.
\end{proof}

\begin{corollary}
Let $X$ be a hyperbolic Riemann surface which is not simply-connected.

Then there exists a non-trivial deformation family of complex structures
on $X$ for which the Kobayashi pseudodistance varies continuously.
\end{corollary}

\begin{proof}
This follows from the preceding corollary if $X$ is isomorphic to
$\Delta^*$ or $\P_1\setminus\{0,1,\infty\}$ and from X
in all other cases.
\end{proof}

Now we can prove the theorem.

\begin{proof}
There is a non-trivial family of complex structures $J_t$ on $X$ such that
the Kobayashi-pseudodistance varies continuously. Let $\tilde X$ denote
the universal covering. We obtain a family of complex structures $\tilde J_t$
such that the Kobayashi pseudodistance varies continuously and the
fundamental group $\pi_1(X)$ acts on $\tilde X$ by ``deck transformations''
which are holomorphic with respect to each $\tilde J_t$.
We fix a point $p\in\tilde X$ and a real tangent direction $\xi$, i.e., a
real half-line in the real tangent space (or equivalently, a tangent
vector of length $1$).
Since $\tilde X$ is hyperbolic and simply-connected, for each $t$ there
is a unique biholomorphic map $\phi_t$ from $(\tilde X,\tilde J_t)$ 
to the unit disc
mapping $p$ to $0$ and $\xi$ to the positive real half line.
The continuity of the Kobayashi pseudodistance implies that these maps
$\phi_t$ as well as there inverse maps $\phi_t^{-1}$ form normal families.
Therefore both $\phi_t$ and $\phi_t^{-1}$ are continuous in $t$.
For each $\gamma\in\pi_1(X)$ let $\alpha_\gamma$ denote the 
deck transformation acting on $\tilde X$. Then 
\[
\pi_1(X)\times I\ni(\gamma,t)\mapsto \phi_t  \circ \alpha_\gamma
\circ \phi_t^{-1} \ \in \Aut(\Delta)\simeq PSL_2(\R)
\]
defines a continuous family of group homomorphisms.
\end{proof}

For arbitrary Lie groups there is a weaker result, cf. \cite{W}.

\begin{proposition}
Let $G$ be a real Lie group which contains a non-compact semisimple
Lie subgroup. Let $k\in\N$. Then there exists a
non-trivial family of injective group homomorphism
$\rho_t$ $(t\in I)$ from the free group $F_k$ with $k$ generators
into $G$ such that $\rho_t(F_k)$ is discrete for every $t$.
\end{proposition}

\begin{proof}
In \cite{W} we proved that there exists a 
non-empty open subset
$W\subset G^k$ such that for every $(g_1,\ldots,g_k)\in W$ the subgroup
of $G$ generated by $g_1,\ldots,g_k$ is free and discrete.
\end{proof}

%\bibliography{rf}
%\bibliographystyle{amsalpha}
\providecommand{\bysame}{\leavevmode\hbox to3em{\hrulefill}\thinspace}
\providecommand{\MR}{\relax\ifhmode\unskip\space\fi MR }
\providecommand{\MRhref}[2]{%
  \href{http://www.ams.org/mathscinet-getitem?mr=#1}{#2}
}
\providecommand{\href}[2]{#2}

\end{document}